\documentclass[11pt]{article}

\usepackage[dvipsnames]{xcolor}

\usepackage{tikz}

\usepackage{cite} 
\usepackage{setspace}
\usepackage[latin1]{inputenc}
\usepackage{amsmath}
 \usepackage{amsthm}
\usepackage{subfig}
\usepackage{amsfonts}
\usepackage{amssymb}
\usepackage{graphicx}
\usepackage{epsfig}
\usepackage{url}
\usepackage{todo}
\usepackage[margin=1.55in]{geometry}
\usepackage{bbm}
\usepackage{microtype}

\usepackage{tikz}
\usepackage[dvipsnames]{xcolor}
\usetikzlibrary{cd}
\usepackage{pgfplots}

\newtheorem{theorem}{Theorem}
\newtheorem{proposition}{Proposition}

\newtheorem{lemma}{Lemma}
\newtheorem{remark}{Remark}

\def\qed{ \rule{.08in}{.08in}}

\newcommand{\diag}{\operatorname{diag}}
\newcommand{\Diag}{\operatorname{Diag}}

\newcommand{\R}{\mathbb{R}}

\newcommand{\U}{\mathrm{Uni}}

\newcommand{\conv}{\operatorname{conv}}
\newcommand{\cX}{\mathcal{X}}

\newcommand{\cA}{\mathcal{A}}

\newcommand{\bfo}{{\bf 1}}

\newcommand{\rL}{{\rm L}}

\newcommand{\cI}{{\cal I}}
\newcommand{\cU}{{\cal U}}

\newcommand{\rint}{\operatorname{int}}

\usepackage{tikz}
\usetikzlibrary{shapes,arrows,graphs,graphs.standard,shapes.misc}
\usetikzlibrary{matrix,decorations.pathreplacing, calc, positioning,fit}
\usetikzlibrary{shapes.geometric}
\usetikzlibrary{decorations.markings}

\newtheorem{definition}{Definition}

\makeatletter
\def\blfootnote{\xdef\@thefnmark{}\@footnotetext}
\makeatother

\def\BibTeX{{\rm B\kern-.05em{\sc i\kern-.025em b}\kern-.08em
    T\kern-.1667em\lower.7ex\hbox{E}\kern-.125emX}}

\begin{document}
\title{On the $H$-property for step-graphons \\and edge polytopes}

\author{M.-A. Belabbas\thanks{M.-A. Belabbas and T.~Ba\c{s}ar are with the Coordinated Science Laboratory, University of Illinois, Urbana-Champaign. Email: \texttt{\{belabbas,basar1\}@illinois.edu.}}
, X. Chen\thanks{X. Chen is with the Department of Electrical, Computer, and Energy Engineering, University of Colorado Boulder. Email: \texttt{xudong.chen@colorado.edu}.}, T.~Ba\c{s}ar$^*$
}
\date{}

\maketitle
\thispagestyle{empty} 
\

\begin{abstract}
\blfootnote{The first two authors contributed equally to the manuscript in all categories.}
Graphons $W$ can be used as stochastic models to sample graphs~$G_n$ on~$n$ nodes for~$n$ arbitrarily large. A graphon~$W$ is said to have the $H$-property if $G_n$ admits a decomposition into disjoint cycles with probability one as $n$ goes to infinity. Such a decomposition is known as a Hamiltonian decomposition.
In this paper, we provide necessary conditions for the $H$-property to hold. The proof builds upon a hereby established  connection between the so-called edge polytope of a finite undirected graph associated with $W$ and the $H$-property. Building on its properties, we provide a purely geometric solution to a random graph problem. 
More precisely, we assign two natural objects to $W$, which we term concentration vector and skeleton graph, denoted by $x^*$ and $S$ respectively. We then establish two necessary conditions for the $H$-property to hold: (1) the edge-polytope of $S$, denoted by $\cX(S)$, is of full rank, and (2) $x^* \in \cX(S)$.
\end{abstract}

\section{Introduction}\label{sec:introduction}
Graphons, a portemanteau of graph and functions, have been recently introduced~\cite{lovasz2006limits, borgs2008convergent} to study very large graphs. A graphon can be understood as both the limit object of a convergent  sequence, where convergence is in the cut-norm~\cite{frieze1999quick}, of graphs of increasing size, and as a statistical model from which to sample random graphs. Taking this latter point of view, we investigate in this paper the so-called $H$-property (see Definition~\ref{def:Hproperty} below) for graphons.

A graphon is a symmetric, measurable function $W: [0,1]^2\to [0,1]$. It gives rise to a stochastic model for undirected graphs on $n$ nodes, denoted by $G_n \sim W$:  
 
\vspace{.2cm}

\noindent 
{\bf Sampling procedure}: Let $\U[0,1]$ be the uniform distribution on $[0,1]$. Given a graphon $W$, a graph $G_n=(V,E)~W$ on  $n$ nodes sampled from $W$  is obtained as follows: 
\begin{enumerate}
    \item Sample $y_1,\ldots,y_n\sim \U[0,1]$ independently. 
    We call $y_i$ the {\em coordinate of node} $v_i\in V$. 

    \item For any two distinct nodes $v_i$ and $v_j$, place an edge $(v_i,v_j) \in E$ with probability $W(y_i,y_j)$.
\end{enumerate}
Note that if $0\leq p\leq 1$ is a {\em constant} and $W(s,t)=p$ for all $(s,t)\in [0,1]^2$, then $G_n \sim W$ is nothing but an Erd\"os-R\'enyi random graph with parameter~$p$. Thus, graphons can be seen, in a sense, as a way to introduce {\em inhomogeneity} of edge densities between different pairs of nodes, and thus increase greatly the type of random graphs one can model. However, all large graphs sampled from (non-zero) graphons have the property of being dense~\cite{lovasz2012large}.
\vspace{.2cm}

\noindent
{\bf $H$-property:}  
Let $W$ be a graphon and $G_n \sim W$. 
In the sequel, we use the notation $\vec G_n=(V,\vec E)$ to denote the {\em directed} version of $G_n$, defined by the edge set
$$\vec E :=\{v_iv_j, v_jv_i \mid (v_i,v_j) \in E \}.$$ 
In words, we replace an undirected edge $(v_i,v_j)$ with two directed edges $v_iv_j$ and $v_jv_i$. 

The directed graph $\vec G_n$ is said to have a {\em Hamiltonian decomposition} if it contains a subgraph $\vec H =  (V, \vec E')$, with the same node set, such that $\vec H$ is a disjoint union of directed cycles. With the preliminaries above, we now have the following definition:  
\vspace{.1cm}

\begin{definition}[$H$-property]\label{def:Hproperty}
Let $W$ be a graphon and $G_n\sim W$. 
Then, $W$ has the $H$-property if 
$$
\lim_{n\to\infty}\mathbb{P}(\vec G_n \mbox{ has a Hamiltonian decomposition}) = 1.
$$
\end{definition}
\vspace{.1cm}

We let $\mathcal{E}_n$ be the event that $\vec G_n$ has a Hamiltonian decomposition. The above definition implicitly requires the sequence $\mathbb{P}(\mathcal{E}_n)$ to converge. We mention here   that for almost all graphons, this sequence converges and, moreover, it converges to either $1$ or $0$. In other words, the $H$-property is a ``zero-one'' property. This fact is, however, beyond the scope of this  paper and will be proven in a forthcoming publication.

The $H$-property is central in the study of structural stability of linear systems~\cite{belabbas_algorithmsparse_2013,belabbas2013sparse} and structural controllability of linear ensemble systems~\cite{chen2021sparse}.  
Indeed, in~\cite{belabbas_algorithmsparse_2013,belabbas2013sparse}, the question of structural stability of linear systems, i.e., of whether a sparsity pattern of matrices contains a stable (Hurwitz) matrix was considered. Using the standard isomorphism between sparsity patterns of square matrices and directed graphs (stemming from interpreting the sparsity pattern as an adjacency matrix) necessary and sufficient conditions were derived on the associated graphs. These conditions required the existence of subgraphs that contained Hamiltonian decompositions. 
In~\cite{chen2021sparse}, the author considered continuum ensembles of sparse linear control systems where the individual systems share a common sparsity pattern, represented by a digraph as above, and characterized the digraphs that can sustain ensemble controllability. 
A complete solution was provided for the case where the parameterization spaces of the ensembles  are closed intervals. In particular, it is shown that the subgraph of the state-nodes needs to have a Hamiltonian decomposition.  

In this paper, we take the first step in our investigation of the $H$-property by focusing on a special class of graphons, which we term step-graphons (the same objects have also been investigated in~\cite{gao2019graphon}). 
Roughly speaking, $W$ is a step-graphon $W$ if one can divide the interval $[0,1]$ into  subintervals $\mathcal{I}_1,\ldots,\mathcal{I}_q$ so that $W$ is constant when restricted to every rectangle  $\mathcal{I}_i\times \mathcal{I}_j$. A more precise definition can be found in Definition~\ref{def:stepgraphon}. Step-graphons are a particular case of the class of step-function graphons introduced in~\cite{lovasz2011finitely}, where the partitioning is into measurable subsets of $[0,1]$. The main contribution of this paper is to obtain {\em necessary} conditions, formulated in Theorem~\ref{thm:main}, for an arbitrary step-graphon to have the $H$-property.  

The key observation underlying the proof is a connection between the $H$-property and polytopes. The study of graphs obtained from polytopes has a long tradition in discrete geometry~\cite{steinitz1928isoperimetrische} but,  later, insights into graph theoretic notions have been obtained from polytopes derived from the  graphs~\cite{ohsugi1998normal}. Our contributions in this paper fall closer to  the latter category: we draw a conclusion about a graphon $W$ from a polytope associated with it. 
The polytope of interest here is the so-called edge polytope~\cite{ohsugi1998normal}; see Definition~\ref{def:edgepolytope}.

This edge polytope appears naturally when seeking characteristics of a step-graphon relevant to whether it has the $H$-property or not.  The first object we exhibit in this vein is the {\em concentration vector} of a step-graphon $W$, denoted by $x^*$, 
the entries of which are the lengths of the subintervals $\mathcal{I}_i$. These entries are also  the probabilities that a random variable $y \sim \operatorname{Uni}[0,1]$ belongs to $\mathcal{I}_i$ (see the first item of the sampling procedure).  
The second object assigned to a step-graphon is its {\em  skeleton graph} $S$, which can be construed as representing the adjacency relations between  rectangles $\mathcal{I}_i\times \mathcal{I}_j$ where the step-graphon is non-zero (see Definition~\ref{def:skeleton}). 
The above mentioned  polytope is then the edge-polytope of the  skeleton graph; we denote it by $\cX(S)$. The two necessary conditions we exhibit in Theorem~\ref{thm:main} are as follows: (1) $S$ has an odd cycle (i.e., a cycle with an odd number of nodes/edges) or, equivalently, $\cX(S)$ has maximal rank,  and (2) $x^*\in \cX(S)$.  

\vspace{.2cm}

\noindent
{\bf Literature review:} 
In recent years, graphons have been used as models for large networks in control and game theory. For control, we mention~\cite{gao2019graphon}; there, the authors consider infinite-dimensional linear control systems $\dot x = A x + B u$, where $x$ and $u$ are elements in $\rL^2([0,1],\R)$ and $A$ and $B$ are bounded linear operators on $\rL^2([0,1],\R)$, obtained by adding scalar multiples of identity operators to graphons. For this class of systems they investigate, among others,  the associated controllability properties and finite-dimensional approximations. For game theory, we mention~\cite{gao2021linear,parise2021analysis} where the authors introduce  different types of graphon games; broadly speaking, these are the games that comprise a continuum of agents (over the closed interval $[0,1]$) with relations between these agents described by a graphon. They then proceed to investigate, among others, the existence of Nash equilibria and properties of finite-dimensional approximations. Finally, the prevalence of the Hamiltonian decompositions was also investigated for Erd\"os-R\'enyi random graphs in~\cite{belabbas2020structural}.

\vspace{.1cm}

\noindent
{\bf Notations and terminology:}   
For $v = (v_1,\ldots, v_n)\in \R^n$, we let $\Diag(v)$ be the $n\times n$ diagonal matrix whose $ii$th entry is $v_i$. We use $\bfo$ to denote the vector whose entries are all ones, and with dimension appropriate for the context.

For $S=(U,F)$, an undirected graph, without multi-edges but possibly with self-loops, we let $\vec S$ be the {\em directed} version of $S$, as defined above, but if  $(u_i,u_i)$ is a self-loop on node $u_i\in U$, then we replace it with a single self-loop $u_iu_i$.

\begin{figure}
    \centering
    \subfloat[\label{sfig1:Ggraph}]{
\includegraphics{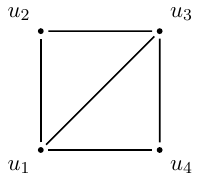}
}
\qquad
\subfloat[\label{sfig1:vecG}]{
\includegraphics{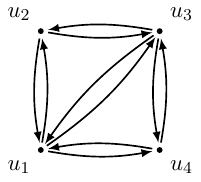}

}
\caption{{\em Left:} An undirected graph $G$ on $4$ nodes with self-loops on $u_2$ and $u_4$. {\em Right:} The directed graph $\vec G$ obtained from $G$. There are several different Hamiltonian decompositions in $\vec G$. For example, the cycle $C_1$ with edge set $\{u_1u_2,u_2u_3,u_3u_4,u_4u_1\}$  forms a Hamiltonian decomposition of $\vec G$. Similarly, the two cycles $C_2$ and $C_3$ with edge sets $\{u_1u_2,u_2u_1\}$ and $ \{u_3u_4,u_4u_3\}$, respectively,  also form a Hamiltonian decomposition of~$\vec G$. }
    \label{fig:hamildecom}
\end{figure}

Given a directed graph $\vec G = (V, \vec E)$ on $n$ nodes without self-loops, the Laplacian matrix $L = [L_{ij}]$ associated with $\vec G$ is the $n\times n$ infinitesimally row stochastic matrix with  off-diagonal entries $L_{ij}$ given by  $L_{ij} = 1$ if $v_iv_j$ is an edge of $\vec G$ and $L_{ij} = 0$ otherwise, and with  diagonal entries picked so that the row sums of $L$ are all $0$, i.e., $L\bfo = 0$.   

For positive integers $\ell, q$, and a set of vectors $z_1,\ldots, z_\ell \in \R^q$, we denote by $\conv\{z_1,\ldots,z_\ell\}$ their {\em convex hull}: 
$$\conv\{z_1,\ldots,z_\ell\}:=\left \{\sum_{i=1}^\ell \lambda_i z_i \mid \sum^\ell_{i = 1}\lambda_i = 1 \mbox{ and }\lambda_i \geq 0\right \}.$$ 

\section{Preliminaries and Main Result}
In this section, we start by defining step-graphons, in Subsection~\ref{ssec:graphons},  and then their associated skeleton graphs and concentration vectors, in Subsection~\ref{ssec:concenandskel}. Then, in Subsection~\ref{ssec:mainresult}, we present  the main result of the paper. 

\subsection{Step-graphons}\label{ssec:graphons}
We have the following definition:
\vspace{.1cm}

\begin{definition}[Step-graphon and its partition]\label{def:stepgraphon}
We call a graphon $W$ a {\bf step-graphon} if there exists an increasing sequence $0 = \sigma_0 < \sigma_1< \cdots < \sigma_q = 1$ such that $W$ is constant over each rectangle $[\sigma_{i}, \sigma_{i + 1})\times [\sigma_{j}, \sigma_{j + 1})$ for all $0\leq i, j\leq q-1$ (there are $q^2$ rectangles in total).  The sequence $ \sigma = (\sigma_0,\sigma_1,\ldots,\sigma_q)$ is called a {\bf partition for $W$}.
\end{definition}
\vspace{.1cm}

\begin{remark}\normalfont
If $W$ is a step-graphon, then there exists an infinite number of compatible partitions  for $W$. 
Indeed, given any partition $\sigma$ for the step-graphon $W$, the partition $\sigma'$ obtained from $\sigma$ by inserting $\sigma'_i$, for any $\sigma_i < \sigma'_i < \sigma_{i+1}$,  is also a partition for $W$.\hfill \qed
\end{remark}
\vspace{.1cm}

We provide an example of a step-graphon in Fig.~\ref{fig:stepgraphon}. Note that a graph $G_n$ sampled from a step-graphon could be seen as a graph sampled from the so-called stochastic block-model~\cite{holland1983stochastic}, but with a random assignment of the nodes to the~$q$ communities with a multinomial distribution determined by the partition sequence. 

Throughout the paper, we let $n_i(G_n)$ be the number of nodes $v_j$ of $G_n$ whose coordinates $y_j \in [\sigma_{i-1},\sigma_i)$ (see item~1 of the sampling procedure in Section~\ref{sec:introduction}). When $G_n$ is clear from the context, we simply write $n_i$.

\begin{figure}
    \centering
    \subfloat[\label{sfig1:stepgraphon}]{
\includegraphics{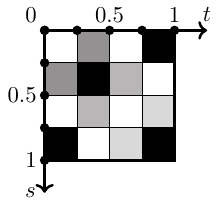}

}
\qquad
\subfloat[\label{sfig1:skeleton}]{
\includegraphics{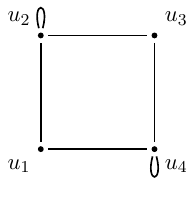}
}
\caption{ {\em Left:} A step-graphon $W$ with the partition sequence $\sigma=(0,0.25,0.5,0.75,1)$. {\em Right:} The associated skeleton graph $S(W)$.}
    \label{fig:stepgraphon}
\end{figure}

\subsection{Concentration vectors and skeleton graphs}\label{ssec:concenandskel}

In this subsection, we introduce  three key objects associated with a step-graphon; namely, its concentration vector, skeleton graph, and the so-called edge polytope of the skeleton graph.   

\vspace{.1cm}

\noindent
{\bf Concentration vector.}  
We have the following definition:
\vspace{.1cm}
\begin{definition}[Concentration vector]
    Let $W$ be  a {step-graphon} with partition $\sigma = (\sigma_0,\ldots,\sigma_q)$. 
    The associated {\bf concentration vector} $x^* = (x^*_1,\ldots, x^*_q)$ has entries  defined as follows: 
    $x^*_i := \sigma_i - \sigma_{i-1}$, for all $i = 1,\ldots, q$. 
\end{definition}
\vspace{.1cm}

There is a one-to-one correspondence between concentration vectors and partition sequences for $W$. 
We further define the  {\em empirical concentration vector} of a graph $G_n \sim W$:
\begin{equation}\label{eq:defecv}x(G_n):= \frac{1}{n}(n_1(G_n),\ldots, n_q(G_n)).
\end{equation}
whose  name is justified by the following observation:  $nx(G_n) = (n_1(G_n),\ldots, n_q(G_n))$ is a multinomial random variable with $n$ trials and $q$ events with probabilities $x^*_i$, for $1 \leq i \leq q$. A straightforward application of Chebyshev's inequality yields that for any $\epsilon>0$, 
\begin{equation}\label{eq:oldlemma1}\mathbb{P}(\|x(G_n)-x^*\| > \epsilon)\leq \frac{c}{n^2\epsilon^2},
\end{equation} 
where $c$ is some constant independent of~$\epsilon$ and~$n$. When $G_n$ and $\sigma$ are clear from the context, we will suppress them and simply write~$x$. 
\vspace{.1cm}

\noindent
{\bf Skeleton graph.} 
A  partition sequence $\sigma$ of a step-graphon $W$ induces a partition of the node set of any $G_n \sim W$ according to which of the intervals $[\sigma_{i-1},\sigma_i)$ the coordinate $y_j$ (of the sampling procedure) of $v_j$ belongs.
Elaborating on this, we can in fact construct a graph which encompasses most of the  relevant characteristics of a step-graphon:
\vspace{.1cm}

\begin{definition}[Skeleton graph]\label{def:skeleton}
To a step-graphon $W$ with a partition sequence $\sigma = (\sigma_0,\ldots, \sigma_q)$, we assign the  undirected graph $S = (U, F)$ on $q$ nodes, with $U =\{u_1,\ldots, u_q\}$ and edge set $F$ defined as follows: there is an edge between $u_i$ and $u_j$ if and only if $W$ is non-zero over $[\sigma_{i-1},\sigma_i)\times [\sigma_{j - 1}, \sigma_j)$. 
We call $S$ the {\bf skeleton graph} of $W$ for the partition sequence $\sigma$.
\end{definition}
\vspace{.1cm}


We decompose the edge set of $S$ as $F=F_0 \cup F_1$, where elements of $F_0$ are self-loops, and elements of $F_1$ are edges between distinct nodes.

Let $\mathcal{I} := \{1,\ldots,|F|\}$ be the index set for~$F$. Let
$\mathcal{I}_0$ (resp. $\mathcal{I}_1$) index the self-loops (resp. edges between distinct nodes) of $S$:  $f_i \in F_0$ for $i \in \cI_0$ (resp. $f_i\in F_1$ for $i \in \cI_1$).

Given a step-graphon $W$ and a skeleton graph $S$, we can naturally introduce a graph homomorphism assigning the nodes of an arbitrary $G_n = (V, E)\sim W$ to their corresponding  nodes in $S$: 
\begin{equation}\label{eq:defpi}
\pi:  v_j\in V \mapsto \pi(v_j) = u_i\in U,
\end{equation}
where $u_i$ is such that $\sigma_{i-1} \leq y_j < \sigma_{i}$, with $y_j$ the coordinate of $v_j$. It should be clear that $n_i(G_n) = |\pi^{-1}(u_i)|$ for all $i = 1,\ldots,q$.

\vspace{0.1cm}

\noindent 
{\bf Edge polytope of a skeleton graph.} To introduce the polytope, we start with the following definition: 
\vspace{.1cm}

\begin{definition}[Incidence matrix]\label{defnodeedgeinc}
Let $S=(U,F)$ be a skeleton graph. Given an arbitrary ordering of its edges and self-loops, we let $Z =[z_{ij}]$ be the associated {\em incidence matrix}, defined as the $|U| \times |F|$ matrix with entries:
\begin{equation}\label{eq:defZS}
z_{ij} := \frac{1}{2}
\begin{cases}
    2, & \text{if } f_j\in F_0 \text{ is a loop on node } u_i,       \\
    1, & \text{if node } u_i \text{ is incident to } f_j\in F_1, \\
    0, & \text{otherwise}.
\end{cases}
\end{equation}\,
\end{definition}

Owing to the factor $\frac{1}{2}$ in~\eqref{eq:defZS}, all columns of $Z$ are probability vectors,
i.e., all entries are nonnegative and sum to one. The edge polytope of $S$ was introduced in~\cite{ohsugi1998normal} and is reproduced below (with slight difference in inclusion of the factor $\frac{1}{2}$ of the generators $z_j$):
\vspace{.1cm}

\begin{definition}[Edge polytope]\label{def:edgepolytope}
Let $S = (U,F)$ be a skeleton graph and $Z$ be the associated incidence matrix. Let $z_j$, for $1\leq j \leq |F|$, be the columns of $Z$.  
The {\em edge polytope} of $S$, denoted by $\cX(S)$, is the finitely generated convex hull: 
\begin{equation}\label{eq:defXS}
\cX(S):= \conv\{z_j \mid j = 1,\ldots, |F|\}.
\end{equation}\,
\end{definition}

Because each $z_j$ is a probability vector and $\cX(S)$ is a convex hull spanned by these vectors, $\cX(S)$ is a subset of the standard simplex in $\R^{q}$. We provide below relevant properties of this set.

We first describe $\cX(S)$ by characterizing its extremal generators. Recall that $x$ is an extremal point of $\cX(S)$ if there is no line segment in $\cX(S)$ that contains $x$ in its interior. Then, the maximal set of extremal points is the set of extremal generators for $\cX(S)$. Because $\cX(S)$ is generated by the columns of $Z$, the set of extremal generators is necessarily a subset of the set of these column vectors. To characterize it further, we let $\cI_2 \subseteq \cI_1$ index the edges of $S$ that are {\em not} incident to two self-loops. We then have
\vspace{.1cm}

\begin{proposition}\label{lem:extremalgeenrator}
The set of extremal generators of $\cX(S)$ is 
$\{z_i \mid i\in \cI_0 \cup \cI_2\}$. 
\end{proposition}

\vspace{.1cm}

\begin{proof}
It should be clear from~\eqref{eq:defZS} that every $z_i$, for $i\in \cI_0$, is an extremal point. Next, note that if $f_i$, for $i\in \cI_1$, is incident to two self-loops, say $f_j$ and $f_k$, then $z_i = \frac{1}{2}(z_j + z_k)$ and, hence, $z_i$ is not extremal.  
It now remains to show that if $i\in \cI_2$, then $z_i$ is an extremal point.  
Suppose not; then, one can write $z_i = \sum_{j\neq i} c_j z_j$, with $c_j \ge 0$. Since $z_i$ only has two non-zero entries and since the $c_j$'s are non-negative, if the support of $z_j$ is not included in the support of $z_i$, then $c_j = 0$. It has two implications: ({\em i}) For any $j\in \cI_1 - \{i\}$, $c_j = 0$; 
({\em ii}) If $j\in \cI_0$, then the self-loop $f_j$ has to be incident to $f_i$.  
Thus, the expression $z_i = \sum_{\ell\neq i} c_\ell z_\ell$ reduces to $z_i =c_{j} z_{j}$, where $f_j$ is the self-loop incident to $z_i$ (if it exists),  which clearly cannot hold. \end{proof}

\vspace{.1cm}

We conclude this subsection with a known result~\cite{ohsugi1998normal} on the rank of $\cX(S)$ (or, similarly, a result~\cite{van1976incidence} on the rank of $Z$ introduced in Definition~\ref{defnodeedgeinc}),  where the rank of $\cX(S)$ is the dimension of its relative interior: 
\vspace{.1cm}

\begin{proposition}\label{prop:rankZSevenodd}
Let $S = (U, F)$ be a connected, undirected graph on $q$ nodes, possibly with loops. Then,  
\begin{equation}
    \operatorname{rank} \cX(S) = 
    \left\{
    \begin{array}{ll}
      q-1  & \mbox{if $S$ has an odd cycle},  \\
      q-2  & \mbox{otherwise}. 
    \end{array}
    \right.
\end{equation}
\end{proposition}

\subsection{Main result}\label{ssec:mainresult}
For ease of exposition,  we assume from now on that the step-graphons $W$ are such that their corresponding skeleton graphs $S$ are connected. However, all the results below hold for step-graphons $W$ whose skeleton graphs have several connected components by requiring that the conditions exhibited for $S$ hold for each connected component of $S$.

\vspace{.1cm}
\begin{theorem}\label{thm:main}
Let $W$ be a step-graphon with $\sigma$ a partition.  
Let $S$ and $x^*$ be the associated (connected) skeleton graph   and concentration vector, respectively. 
Let $G_n\sim W$ and $\vec G_n$ be the directed version of $G_n$. 
If $S$ has no odd cycle or if $x^*\notin \cX(S)$, then 
\begin{equation}\label{eq:nonHproperty}
\lim_{n\to\infty}\mathbb{P}(\vec G_n \mbox{  has a Hamiltonian decomposition}) = 0.
\end{equation}
\end{theorem}

\vspace{.3cm}

The proof goes by showing that if one of the two conditions holds, then the edge polytope $\cX(S)$ contains {\em at most} a zero-measure subset of the support of~$x^*$. In particular, if $S$ does not have an odd cycle, then the codimension of $\cX(S)$ in the standard simplex is one (see Proposition~\ref{prop:rankZSevenodd}) and thus the probability that the vector $x^\star$ belongs to $\cX(S)$ is negligible in the asymptotic regime.

The conditions exhibited in Theorem~\ref{thm:main}  almost completely determine whether $W$ has the $H$-property:  we can show that $S$  has an odd cycle and  $x^*$ is in the {\em interior} of  $\cX(S)$, then   
$$\lim_{n\to\infty}\mathbb{P}(\vec G_n \mbox{  has a Hamiltonian decomposition}) = 1.$$
The proof of this statement is much more involved than the proof of Theorem~\ref{thm:main}, and will be presented on another occasion.

It may seem at first that the main result depends on a certain partition $\sigma$, which defines $S$ and $x^*$. We state here the following fact: 

\begin{proposition}\label{prop:invarianceofskeleton}
Let $W$ be a step-graphon. For any two partitions $\sigma$ and $\sigma'$ for $W$, let $x^*$, $x'^*$ be the corresponding concentration vectors and let $S$, $S'$ be the corresponding skeleton graphs. Then, the following hold:
\begin{enumerate}
    \item $S$ is connected if and only if $S'$ is connected;
    \item $S$ has an odd cycle if and only if $S'$ has an odd-cycle;
    \item $x^*\in \cX(S)$ (resp. $x\in\rint \cX(S)$) if and only if $x'^*\in \cX(S')$ (resp. $x'^*\in \cX(S')$).
\end{enumerate}
\end{proposition} 

The proof of the result is provided in the Appendix.

\section{Analysis and Proof of Theorem~\ref{thm:main}}
\subsection{On the edge polytope of $S$}

Let $W$ be a step-graphon with partition sequence $\sigma$ and corresponding skeleton graph $S$ on $q$ nodes. 
In this subsection, we  introduce  in Definition~\ref{def:CaS} the set $\cA(S)$ of sparse infinitesimally stochastic matrices whose sparsity pattern is determined by the skeleton graph $S$.  We then show that the edge polytope $\cX(S)$, defined by~\eqref{eq:defXS}, is exactly the set of row sums of these matrices. This expression of $\cX(S)$ will be required in the proof of Theorem~\ref{thm:main}.

We start with the following result:

\vspace{.1cm}
\begin{lemma}\label{lem:doublestochastic}
Assume that $\vec G_n$ has a Hamiltonian decomposition, denoted by $\vec H$, and let $n_{ij}(\vec H)$ be the number of edges of $\vec H$ from a node in $\pi^{-1}(u_i)$ to a node in $\pi^{-1}(u_j)$.
Then, for all $u_i \in U$, 
\begin{equation}\label{eq:solH}
n_i(G_n)=\sum_{u_j \in N(u_i)} n_{ij}(\vec H) = \sum_{u_j\in N(u_i)} n_{ji}(\vec H).
\end{equation}\,
\end{lemma}

\begin{proof}
Each node of $\vec H$ has exactly one incoming edge and one outgoing edge. The result then follows from the fact that $\sum_{u_j \in N(u_i)} n_{ij}(\vec H)$ counts the number of outgoing edges from the nodes of $\pi^{-1}(u_i)$ while $\sum_{u_j \in N(u_i)} n_{ji}(\vec H)$ counts the number of incoming edges to the nodes of $\pi^{-1}(u_i)$, and the fact that $\vec H$ has the same node set as $\vec G_n$.
\end{proof}
\vspace{.1cm}

Following Lemma~\ref{lem:doublestochastic}, we  now assign to the skeleton graph $S$ a convex set that will be instrumental in the study of Hamiltonian decompositions of $\vec G_n$: 
\vspace{.1cm}

\begin{definition}\label{def:CaS}
To an arbitrary undirected graph $S = (U,F)$ on $q$ nodes, possibly with self-loops, we assign the set $\cA(S)$ of $q\times q$ nonnegative matrices $A = [a_{ij}]$ that satisfy the following two conditions: 
\begin{enumerate} 
\item if $(u_i,u_j) \notin F$, then $a_{ij}=0$;
\item $A\bfo = A^\top \bfo$, and $\bfo^\top A \bfo = 1$. 
  \end{enumerate} 
\end{definition}
\vspace{.1cm}

Note that $\bfo^\top A \bfo$ is nothing but the sum of all the entries of $A$. 
Because every defining condition for ${\cal A}(S)$ is affine, the set $\cA(S)$ is a convex set.   

Now, to each Hamiltonian decomposition $\vec H$ of $\vec G_n$,  we assign the following $q\times q$ matrix:  
\begin{equation}\label{eq:IDSM}
\rho(\vec H):= \frac{1}{n} \left[ n_{ij}(\vec H)\right]_{1\leq i,j\leq q}.
\end{equation}
The next lemma then follows immediately from Lemma~\ref{lem:doublestochastic}: 
\vspace{.1cm}

\begin{lemma}\label{lem:phiHinconvS}
If $\vec H$ is a Hamiltonian decomposition of $\vec G_n$, then $\rho(\vec H) \in \cA(S)$ and $\rho(\vec H)\bfo = x$, where $x$ is the empirical concentration vector of $G_n$. 
\end{lemma}
\vspace{.1cm}

The relation $\rho(\vec H) \bfo= x$ in the above lemma leads us to investigate the set of the possible row sums of $A \in \cA(S)$. The main result of this subsection is that this set is  equal to~$\cX(S)$ introduced in~\eqref{eq:defXS}: 
\vspace{.1cm}

\begin{proposition}\label{prop:charDSZ}
The following holds: 
\begin{equation}\label{eq:ds}
\cX(S) = \{ x\in \R^q \mid x=A\bfo \mbox{ for some } A\in \cA(S) \}.
\end{equation}
\end{proposition}
\vspace{.2cm}

\begin{proof}
We prove the result by using double-inclusion: 

\vspace{.1cm}

\noindent
{\em 1. Proof that $\cX(S)\subseteq \cA(S)\bfo$.} 
We show that for each generator $z_j$ of $\cX(S)$ as in~\eqref{eq:defZS}, there exists an $A\in \cA(S)$ such that $z_j = A\bfo$.  
If $j\in \cI_0$, then $f_j$ is a loop on some node $u_i$. 
Let $A_j:=e_ie_i^\top\in \cA(S)$; then, $A_j\bfo = z_j$.  
If $j\in \cI_1$, then $f_j = (u_k,u_\ell)$ is an edge between two distinct nodes.  
Let $A_j := \frac{1}{2} (e_k e_\ell^\top + e_\ell e_k^\top)\in \cA(S)$; then, $A_j\bfo = z_j$. 
\vspace{.1cm}

\noindent
{\em 2. Proof that $\cX(S)\supseteq\cA(S)\bfo$.} 
Let $A\in \cA(S)$, and we show that $A\bfo \in \cX(S)$. By Definition~\ref{def:CaS}, $A\bfo$ belongs to the standard simplex. Thus, it suffices to show that $A\bfo$ can be written as a nonnegative combination  of the $z_j$'s; indeed, if this holds, then it has to be a convex combination of the $z_j$'s and, hence, $A\bfo \in \cX(S)$. 

Decompose $A=:A_0 + A_1$ where $A_0$ (resp. $A_1$) is the diagonal (resp. off-diagonal) part of $A$. Then, $A\bfo = A_0\bfo + A_1\bfo$. We show that both $A_0\bfo$ and $A_1\bfo$ can be written as nonnegative combinations of $z_j$'s.  

For $A_0\bfo$, note that if the $ii$th entry of $A_0$ is not $0$, then~$u_i$ has a self-loop, say $f_j$. Thus, we obtain that $A_0\bfo$ can be expressed as a nonnegative combination of $z_j$'s, for $j \in \cI_0$. 

For $A_1\bfo$, we translate the problem into a problem about decompositions of infinitesimally doubly stochastic matrices into Laplacian matrices of cycles. First, since $A_1\bfo = A_1^\top\bfo$, {\em replacing} the diagonal entries of $A_1$ with the entries of $-A_1\bfo$ results in an infinitesimally {\em doubly} stochastic matrix. We denote it by $A'_1$ (i.e., $A'_1:=A_1-\Diag(A_1\bfo)$).

Now, consider the directed version of $S$, denoted by $\vec S$. For each directed cycle $\vec C_k$ of $\vec S$, other than self-loops, we let $L'_k$ be the associated Laplacian matrix. 
It is known that $A'_1$ can be expressed as a nonnegative combination of these $L'_k$~\cite[Proposition~3]{chen2016distributed} (the statement can be viewed as an infinitesimal version of the Birkhoff Theorem~\cite{birkhoff1946tres} for doubly stochastic matrices). 
In particular, the diagonal of $A'_1$ (which is $-A_1\bfo$) is a nonnegative combination of the diagonals of $L'_k$. 
Hence, it remains to show that the diagonal of each $L'_k$ can be written as a nonpositive combination of the $z_j$'s, for $j\in \cI_1$. 
Let $j_1,\ldots, j_m$ be the indices in $\cI_1$ that correspond to the undirected versions of the edges of $\vec C$. Then, 
$-\diag(L'_k) = \sum^m_{\ell = 1} z_{j_\ell}$.  This completes the proof.
\end{proof}

\vspace{.1cm}

\subsection{Proof of Theorem~\ref{thm:main}}
Let $G_n\sim W$ and $x$ be the associated empirical concentration vector. We will address subsequently the two conditions (1) $x^* \notin \cX(S)$ and (2) $S$ having no odd cycle: 
\vspace{.1cm}

\noindent
{\em Condition (1): $x^* \notin \cX(S)$}. 
Since $\cX(S)$ is closed, if $x^*\notin \cX(S)$, then there is an open neighborhood $\cU$ of $x^*$ in the standard simplex such that $\cU \cap \cX(S) = \varnothing$. 
On the one hand, by~\eqref{eq:oldlemma1}, the probability that $x$ belongs to $\cU$ tends to $1$ as $n$ goes to infinity. 
On the other hand, if $\vec G_n$ admits a Hamiltonian decomposition $\vec H$, then by Lemma~\ref{lem:phiHinconvS}, $A:=\rho(\vec H)\in \cA(S)$ and $x = A\bfo\in \cX(S)$. The above arguments imply that if $x^* \notin \cX(S)$, then~\eqref{eq:nonHproperty} holds.   
\vspace{.1cm}

\noindent
{\em Condition (2): $S$ has no odd cycle}.
In this case, by the definition of $\cX(S)$ in~\eqref{eq:defXS} and Proposition~\ref{prop:rankZSevenodd}, the co-dimension of $\cX(S)$ is~$1$ in the standard simplex. We introduce the random variable $\omega_n:=\sqrt{n} (x-x^*) + x^*$. Since $\mathbb{E}x= x^*$, it is known~\cite{arenbaev1977asymptotic} that $\omega_n$ converges in law to a Gaussian random variable~$\omega$ with mean $x^*$ and  covariance $\Sigma:=\Diag(x^*) - x^*x^{*\top}$. A short calculation yields that $\Sigma\bfo=0$ and that $\Sigma$ has rank $(q-1)$ (one could see this by, e.g., relating it to a weighted Laplacian matrix of a complete graph). 
Hence, the support of $\omega$ is the affine hyperplane $Q:=\{x^* + v\mid v^\top \bfo = 0\}$. 
Next, let $Q'\subsetneq Q$ be the smallest affine hyperplane containing $\cX(S)$, i.e., $Q':=\{ \sum_{i = 1}^{|F|} \lambda_i z_i \mid \sum^{|F|}_{i = 1} \lambda_i = 1 \}$. Its co-dimension in $Q$ is~$1$, so $\mathbb{P}(\omega\in Q') = 0$.  
Since $\omega_n$ converges in law to $\omega$, $\lim_{n\to\infty}\mathbb{P}(\omega_n\in Q') = 0$. 
We conclude the proof by noting that the event $\omega_n \in Q'$ is necessary for $x\in \cX(S)$ and, hence, by Lemma~\ref{lem:phiHinconvS}, necessary for $G_n$ to have a Hamiltonian decomposition. 
\hfill\qed
 
\section{Numerical Validations}\label{sec:numvalidation}

\begin{figure}
    \centering
    \subfloat[\label{sfig1:stepgraphon1}]{
\includegraphics{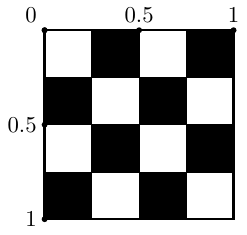}
} 
\subfloat[\label{sfig1:stepgraphon2}]{
\includegraphics{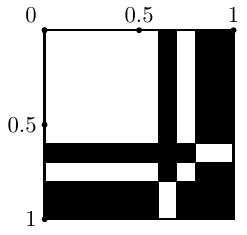}
}

\subfloat[\label{sfig1:stepgraphon3}]{
\includegraphics{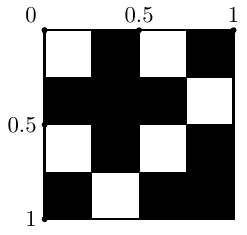}
}
\subfloat[\label{sfig1:stepgraphon4}]{
\includegraphics{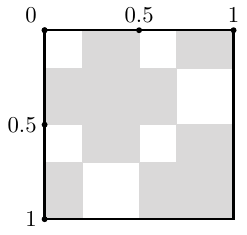}
}

\subfloat[\label{sfig1:skeleton1}]{
 \includegraphics{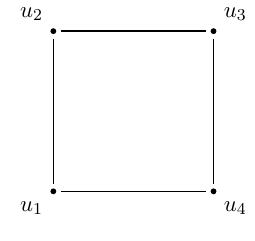}
}
\subfloat[][\label{sfig1:skeleton2}]{
 \includegraphics{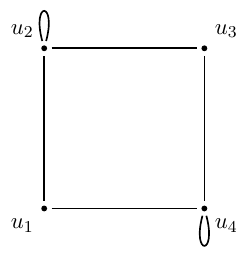}
}

\caption{The step-graphon depicted in Fig.~\ref{sfig1:stepgraphon1} has the skeleton graph shown in Fig.~\ref{sfig1:skeleton1}; the step-graphons depicted in Figs.~\ref{sfig1:stepgraphon2},~\ref{sfig1:stepgraphon3},~\ref{sfig1:stepgraphon4} have the skeleton graph shown in Fig.~\ref{sfig1:skeleton2}. The value of $W(s,t)$ is color-coded, with black being $1$ and white being $0$. The gray value in step-graphon~\ref{sfig1:stepgraphon4} corresponds to $0.2$. For each of these step-graphons, we sampled $N=2\cdot 10^4$ graphs $G_n$ for various $n$ and evaluated whether $\vec G_n$ have a Hamiltonian decomposition. The results are shown in Fig.~\ref{fig:simresult}.}\label{fig:stepgraphon2}
\end{figure}

We performed numerical studies to understand how rapidly the asymptotic regime appears as $n$ grows larger. The simulation results can also be understood as a validation of our main theorem and the claims made in the paper. The set-up is the following: we consider the four step-graphons depicted in Fig.~\ref{fig:stepgraphon2}. For each step-graphon, we sampled sets of $N=2 \cdot 10^4$ graphs $G_n$ for each $n \in \{10,20,50,100,250,500,1000\}$ and evaluated the proportion of $\vec G_n$ that have a Hamiltonian decomposition. Namely, we evaluated
$$
P_H := \frac{\mbox{\#} \vec G_n \mbox{ with Hamiltonian decomposition}}{2 \cdot 10^4}.
$$
Below are the observations from the experiments:
\vspace{.1cm}

\noindent
{\em Experiment (a)}: The step-graphon shown in Fig.~\ref{sfig1:stepgraphon1} has associated concentration vector $x^*=[0.25,0.25,0.25,0.25]$. Its skeleton graph $S$, shown in Fig.~\ref{sfig1:skeleton1}, does not have an odd cycle, but $x^*\in \cX(S)$.  
We observe in Fig.~\ref{fig:simresult} that the proportion of $\vec G_n \sim W$ that contains a Hamiltonian decomposition goes to zero as $n \to \infty$.
\vspace{.1cm}

\noindent
{\em Experiment (b)}: 
The step-graphon shown in Fig.~\ref{sfig1:stepgraphon2} has associated concentration vector $x^*=[0.6,0.1,0.1,0.2]$. The skeleton graph $S$, shown in Fig.~\ref{sfig1:skeleton2}, has an odd cycle. However, $x^* \notin \cX(S)$. We observe in Fig.~\ref{fig:simresult} that the proportion of $\vec G_n \sim W$ that contains a Hamiltonian decomposition goes to zero as $n \to \infty$.
\vspace{.1cm}

\noindent
{\em Experiment (c)}: 
The step-graphon shown in Fig.~\ref{sfig1:stepgraphon3} has associated concentration vector $x^*=[0.25,0.25,0.25,0.25]$. The skeleton graph $S$, shown in Fig.~\ref{sfig1:skeleton2}, has an odd cycle. One can check that $x^* \in \partial \cX(S)$, i.e., the {\em boundary} of $\cX(S)$.  We observe in Fig.~\ref{fig:simresult} that the proportion of $\vec G_n \sim W$ that contains a Hamiltonian decomposition does not vanish as $n \to \infty$ nor goes to $1$. Note that the class of step-graphons such that $x^* \in \partial \cX(S)$ is not generic. 
\vspace{.1cm}

\noindent
{\em Experiment (d)}: 
The step-graphon shown in Fig.~\ref{sfig1:stepgraphon4} has associated concentration vector $x^*=[0.25,0.25,0.25,0.25]$. The skeleton graph $S$, shown in Fig.~\ref{sfig1:skeleton2}, has an odd cycle. One can check that $x^* \in \rint \cX(S)$, the interior of $\cX(S)$.  We observe in Fig.~\ref{fig:simresult} that the proportion of $\vec G_n \sim W$ that contain a Hamiltonian decomposition converges to~$1$ as $n \to \infty$. 

\begin{figure}
    \centering
\includegraphics{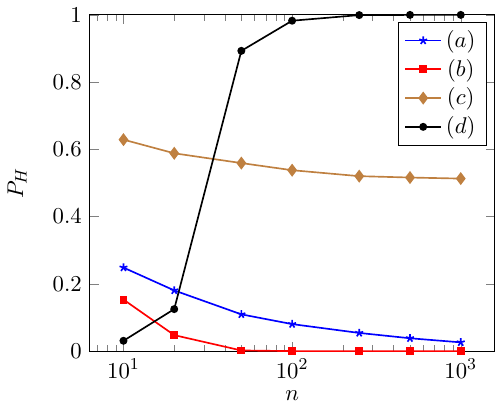}

    \caption{We plot the proportion $P_H$ of $\vec G_n$, with $G_n \sim W$, that have a Hamiltonian decomposition, for the four step-graphons depicted in Fig.~\ref{fig:stepgraphon2}.}
    \label{fig:simresult}
\end{figure}

\section{Conclusions}
We have exhibited two necessary conditions for the $H$-property to hold for the class of step-graphons $W$. The starting point of our analysis was the introduction of two novel objects associated with $W$: its concentration vector $x^*$ and its skeleton graph $S$. We have then  highlighted a novel connection between the edge polytope of $S$, denoted by $\cX(S)$, and the $H$-property for the underlying graphon $W$: it requires that $x^*\in \cX(S)$ and $\cX(S)$ is of maximal rank. We also validated our results via numerical studies in Sec.~\ref{sec:numvalidation}. As was claimed after Theorem~\ref{thm:main} and shown in Figure~\ref{fig:simresult}, the two conditions that  $x^*$ belongs to the interior of $\cX(S)$ and that $\cX(S)$ is of maximal rank are sufficient for a step-graphon $W$ to have the $H$-property.

\bibliographystyle{IEEEtran}      
\bibliography{refs.bib}       

\begin{thebibliography}{10}
\providecommand{\url}[1]{#1}
\csname url@samestyle\endcsname
\providecommand{\newblock}{\relax}
\providecommand{\bibinfo}[2]{#2}
\providecommand{\BIBentrySTDinterwordspacing}{\spaceskip=0pt\relax}
\providecommand{\BIBentryALTinterwordstretchfactor}{4}
\providecommand{\BIBentryALTinterwordspacing}{\spaceskip=\fontdimen2\font plus
\BIBentryALTinterwordstretchfactor\fontdimen3\font minus
  \fontdimen4\font\relax}
\providecommand{\BIBforeignlanguage}[2]{{%
\expandafter\ifx\csname l@#1\endcsname\relax
\typeout{** WARNING: IEEEtran.bst: No hyphenation pattern has been}%
\typeout{** loaded for the language `#1'. Using the pattern for}%
\typeout{** the default language instead.}%
\else
\language=\csname l@#1\endcsname
\fi
#2}}
\providecommand{\BIBdecl}{\relax}
\BIBdecl

\bibitem{lovasz2006limits}
L.~Lov{\'a}sz and B.~Szegedy, ``Limits of dense graph sequences,''
  \emph{Journal of Combinatorial Theory, Series B}, vol.~96, no.~6, pp.
  933--957, 2006.

\bibitem{borgs2008convergent}
C.~Borgs, J.~T. Chayes, L.~Lov{\'a}sz, V.~T. S{\'o}s, and K.~Vesztergombi,
  ``Convergent sequences of dense graphs i: Subgraph frequencies, metric
  properties and testing,'' \emph{Advances in Mathematics}, vol. 219, no.~6,
  pp. 1801--1851, 2008.

\bibitem{frieze1999quick}
A.~Frieze and R.~Kannan, ``Quick approximation to matrices and applications,''
  \emph{Combinatorica}, vol.~19, no.~2, pp. 175--220, 1999.

\bibitem{lovasz2012large}
L.~Lov{\'a}sz, \emph{Large Networks and Graph Limits}.\hskip 1em plus 0.5em
  minus 0.4em\relax American Mathematical Soc., 2012, vol.~60.

\bibitem{belabbas_algorithmsparse_2013}
M.-A. Belabbas, ``Algorithms for sparse stable systems,'' in \emph{Proceedings
  of the 52th IEEE Conference on Decision and Control}, 2013.

\bibitem{belabbas2013sparse}
------, ``Sparse stable systems,'' \emph{Systems \& Control Letters}, vol.~62,
  no.~10, pp. 981--987, 2013.

\bibitem{chen2021sparse}
X.~Chen, ``Sparse linear ensemble systems and structural controllability,''
  \emph{IEEE Transactions on Automatic Control}, 2021, appeared online.

\bibitem{gao2019graphon}
S.~Gao and P.~E. Caines, ``Graphon control of large-scale networks of linear
  systems,'' \emph{IEEE Transactions on Automatic Control}, vol.~65, no.~10,
  pp. 4090--4105, 2019.

\bibitem{lovasz2011finitely}
L.~Lov{\'a}sz and B.~Szegedy, ``Finitely forcible graphons,'' \emph{Journal of
  Combinatorial Theory, Series B}, vol. 101, no.~5, pp. 269--301, 2011.

\bibitem{steinitz1928isoperimetrische}
E.~Steinitz, ``{\"U}ber isoperimetrische probleme bei konvexen polyedern.''
  1928.

\bibitem{ohsugi1998normal}
H.~Ohsugi and T.~Hibi, ``Normal polytopes arising from finite graphs,''
  \emph{Journal of Algebra}, vol. 207, no.~2, pp. 409--426, 1998.

\bibitem{gao2021linear}
S.~Gao, R.~F. Tchuendom, and P.~E. Caines, ``Linear quadratic graphon field
  games,'' \emph{Communications in Information and Systems}, vol.~21, pp.
  341--369, 2021.

\bibitem{parise2021analysis}
F.~Parise and A.~Ozdaglar, ``Analysis and interventions in large network
  games,'' \emph{Annual Review of Control, Robotics, and Autonomous Systems},
  vol.~4, pp. 455--486, 2021.

\bibitem{belabbas2020structural}
M.-A. Belabbas and A.~Kirkoryan, ``On the structural stability of random
  systems,'' \emph{arXiv preprint arXiv:2003.04139}, 2020.

\bibitem{holland1983stochastic}
P.~W. Holland, K.~B. Laskey, and S.~Leinhardt, ``Stochastic blockmodels: First
  steps,'' \emph{Social Networks}, vol.~5, no.~2, pp. 109--137, 1983.

\bibitem{van1976incidence}
C.~Van~Nuffelen, ``On the incidence matrix of a graph,'' \emph{IEEE
  Transactions on Circuits and Systems}, vol.~23, no.~9, pp. 572--572, 1976.

\bibitem{chen2016distributed}
X.~Chen, M.-A. Belabbas, and T.~Ba{\c{s}}ar, ``Distributed averaging with
  linear objective maps,'' \emph{Automatica}, vol.~70, pp. 179--188, 2016.

\bibitem{birkhoff1946tres}
G.~Birkhoff, ``Tres observaciones sobre el algebra lineal,'' \emph{Univ. Nac.
  Tucuman, Ser. A}, vol.~5, pp. 147--154, 1946.

\bibitem{arenbaev1977asymptotic}
N.~Arenbaev, ``Asymptotic behavior of the multinomial distribution,''
  \emph{Theory of Probability \& Its Applications}, vol.~21, no.~4, pp.
  805--810, 1977.

\end{thebibliography}

\appendix
\section{Analysis and Proof of Proposition~\ref{prop:invarianceofskeleton}}
We first have some preliminaries about {\em refinements} of partitions: given a partition sequence $\sigma$, a {\em refinement} $\sigma'$ of $\sigma$, denoted by $\sigma\prec \sigma'$, is any sequence that has $\sigma$ as a proper subsequence. For example, $\sigma' = (0,1/2,3/4,1)$ is a refinement of $\sigma = (0,1/2,1)$. 
Given a step-graphon $W$, if $\sigma$ is a partition for $W$, then so is $\sigma'$.

We say that $\sigma'$ is a {\em one-step refinement} of $\sigma$ if it is a refinement with $|\sigma'|=|\sigma|+1$. Any refinement of $\sigma$ can be obtained by iterating one-step refinements. To fix ideas,  and without loss of generality, we consider the refinement of  $\sigma=(\sigma_0,\ldots, \sigma_q,\sigma_*)$ to $\sigma'=(\sigma_0,\ldots, \sigma_q,\sigma_{q+1}, \sigma_*)$ with $\sigma_{q} < \sigma_{q+1} < \sigma_*$. 
If $S = (U,F)$, then $S'=(U',F')$, the skeleton graph of $W$ for $\sigma'$, is given by
\begin{equation}\label{eq:splitnode}
\left\{
\begin{aligned}    
    U' &= U \cup \{ u_{q+1}\}, \\
    F' &= F \cup \{ (u_i, u_{q+1}) \mid (u_i,u_q)\in F \} \cup \{(u_{q+1}, u_{q+1}) \mbox{ if } (u_q, u_q)\in F\}.
\end{aligned}
\right.
\end{equation}
In essence, the node $u_{q+1}$ is a copy of the node $u_q$. If there is a  loop $(u_q,u_q)$ in $F$, then $u_q$ and $u_{q+1}$ are also connected and each has a self-loop. See Fig.~\ref{fig:split} for illustration. 
We say that a one-step refinement {\em splits a node} (here, $u_q$).

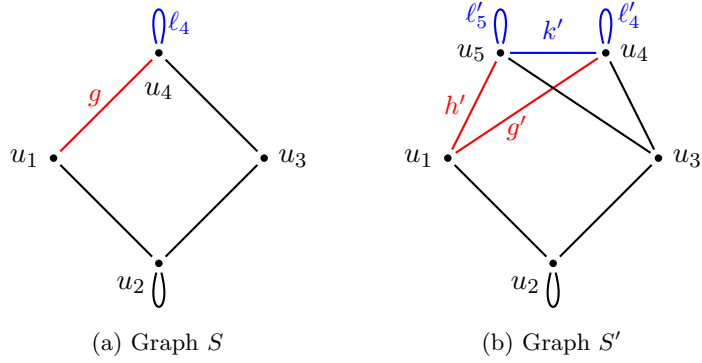
\begin{figure}
    \centering
    \subfloat[\label{sfig1:beforesplit}][Graph $S$]{

 \begin{tikzpicture}[scale=1.4]
    \tikzset{every loop/.style={}}
		\node [circle,fill=black,inner sep=1pt,label=left:{$u_1$}] (1) at (-1, 0) {};
		\node [circle,fill=black,inner sep=1pt,label=below left:{$u_2$}] (2) at (0, -1) {};
		\node [circle,fill=black,inner sep=1pt,label=right:{$u_3$}] (3) at (1, 0) {};
		\node [circle,fill=black,inner sep=1pt,label={[label distance=.2cm]-90:$u_4$}] (4) at (0, 1) {};

  \path[draw,thick,shorten >=2pt,shorten <=2pt]
		 (2) edge[loop below] (2)
		 (4) edge[loop above, blue] (4) node[  label={[font=\small,label distance=0.1cm,text=blue, xshift=-.25cm]25:$\ell_4$}] {}
		 (1) edge[black] (2)
		 (1) edge[red] (4) node[  label={[font=\small,label distance=0.5cm,text=red, xshift=-.07cm]55:$g$}] {}
		 (2) edge[black] (3)
		 (3) edge[black] (4)
		 ;
\end{tikzpicture}

}
\qquad
\subfloat[\label{sfig1:aftersplit}][Graph $S'$]{
   \begin{tikzpicture}[scale=1.4]
    \tikzset{every loop/.style={}}
 \tikzset{every loop/.style={}}
		\node [circle,fill=black,inner sep=1pt,label=left:{$u_1$}] (1) at (-1, 0) {};
		\node [circle,fill=black,inner sep=1pt,label=below left:{$u_2$}] (2) at (0, -1) {};
		\node [circle,fill=black,inner sep=1pt,label=right:{$u_3$}] (3) at (1, 0) {};
		\node [circle,fill=black,inner sep=1pt,label=right:{$u_4$}] (4) at (0.5, 1) {};
		\node [circle,fill=black,inner sep=1pt,label=left:{$u_5$}] (5) at (-0.5, 1) {};

  \path[draw,thick,shorten >=2pt,shorten <=2pt]
		 (2) edge[loop below] (2)
		 (4) edge[loop above,blue] (4) node[label={[font=\small,label distance=0.1cm,text=blue, xshift=-.2cm]40:$\ell'_4$}] {}
		 (5) edge[loop above,blue] (5) node[label={[font=\small,label distance=0.1cm,text=blue, xshift=.2cm]140:$\ell'_5$}] {}
		 (1) edge[black] (2)
		 (1) edge[red] (4) node[label={[font=\small,label distance=0.4cm,text=red,xshift=.1cm]10:$g'$}] {}
		 (1) edge[red] (5) node[  label={[font=\small,label distance=0.3cm,text=red, xshift=.1cm]90:$h'$}] {}
		 (2) edge[black] (3)
		 (3) edge[black] (4)
		 (3) edge[black] (5)
		 (5) edge[blue] (4) node[label={[font=\small,xshift=.7cm,yshift=-.1cm,text=blue]:$k'$}] {}
		 ;
\end{tikzpicture} 
}
\caption{ The graph $S'$ on the right is obtained from the left $S$ via a one-step refinement. The node $u_4$ in $S$ is split into $u_4$ and $u_5$. 
Because  $g=(u_1,u_4)$ is an edge in $S$, there exist two edges $g'=(u_1,u_4)$ and $h'=(u_1,u_5)$ in $S'$. Because $u_4$ has a self-loop $\ell_4$ in $S$, both $u_4$ and $u_5$ have self-loops in $S'$, denoted by $\ell'_4$ and $\ell'_5$, respectively. In addition, we have the edge $k'=(u_4,u_5)$ in $S'$.}
    \label{fig:split}
\end{figure}

We now prove Proposition~\ref{prop:invarianceofskeleton}: 

\begin{proof}[Proof of Proposition~\ref{prop:invarianceofskeleton}] 
Let $\sigma$ and $\sigma'$ be as given in the statement of the proposition. 
It should be clear that there exists another partition $\sigma''$ which is a refinement of both $\sigma'$ and $\sigma$ and that $\sigma''$ can be obtained via a sequence of one-step refinements starting with either $\sigma'$ or $\sigma$. Thus, combining the arguments at the beginning of the section, we can assume, without loss of generality, that $\sigma'$ is a {\em one-step} refinement of $\sigma$ obtained by splitting the node $u_1 \in U$. 

Let $x^*$ and $x'^*$ be the concentration vectors for $\sigma$ and $\sigma'$, $S$ and $S'$ be the corresponding skeleton graphs, and $Z$ and $Z'$ be the corresponding incidence matrices. 
Note that $Z'$ has one more row than $Z$ does due to the addition of the new node $u_{q+1}$; here, we let the last row of $Z'$ correspond to that node. It should be clear that $Z'$ contains $Z$ as a submatrix. 
For clarity of presentation, we use $f$ (resp. $f'$) to denote edges of $S$ (resp. $S'$). Since the graph $S$ can be realized as a subgraph of $S'$ in a natural way, we will write on occasion $f'\in F$ if $f'$ is an edge of $S$.  

We now prove the invariance of each item listed in the statement of Proposition~\ref{prop:invarianceofskeleton} under one-step refinements.
The proofs of the first two items are direct consequence of the definition of one-step refinement.
\vspace{.2cm}

\noindent {\bf Proof for item (1).} If $S$ is connected, then  from~\eqref{eq:splitnode} we obtain that there exists a path from any node $u_i \in F$ to the new node $u_{q+1}$, so $S'$ is also connected. Reciprocally, assume that $S$ has at least two connected components. 
Then, the node $u_{q+1}$ obtained by splitting $u_q$ will only be connected to nodes in the same component as $u_q$ by definition of $F'$. 
\vspace{.2cm}

\noindent {\bf Proof for item (2).} If $S$ has an odd cycle, then so does $S'$ by~\eqref{eq:splitnode}. Reciprocally, we assume that $S$ is lacking an odd cycle. 
We show that $S'$ has no odd cycle. Suppose, to the contrary, that it does.  The cycle must then contain the node $u_{q+1}$. Replacing $u_{q+1}$ with $u_q$ yields a closed walk of odd length in $S$.  Since a closed walk can be decomposed edge-wise into a union of cycles and since the length of the walk is the sum of the lengths of the constituent cycles, there must exist an odd cycle in $S$, which is a contradiction.
\vspace{.2cm}

\noindent {\bf Proof for item (3)} We prove each direction of the statement separately:
\vspace{.1cm}

\noindent{\bf Part 1: $x\in \cX(S) \Rightarrow x'\in \cX(S')$ ($x\in \rint \cX(S) \Rightarrow x'\in \rint \cX(S')$). }
For ease of presentation, we let $z_f$ (resp. $z'_{f'}$) be the edge of $Z$ (resp. $Z'$) corresponding to the element $f\in F$ (resp. $f'\in F'$), and $z_{f,i}$ be the $i$th entry of $z_f$. 
Because $\cX(S)$ is the convex hull of the columns of $Z$,  there exist coefficients $c_f\geq 0$, for $f\in F$, such that $x= \sum_{f\in F} c_f z_f$. 
If, further, $x\in \rint \cX(S)$, then these coefficients can be chosen to be strictly positive. 
We will use $c_f$ to construct $c'_{f'} \ge 0$, for $f' \in F'$, such that 
\begin{equation}\label{eq:x'c'f'z'}
x'=\sum_{f' \in F'} c'_{f'} z'_{f'}
\end{equation}
and show that $x'\in \rint \cX(S')$ if $x\in \rint \cX(S)$.  

To proceed, let $F_{u_q}$ be the set of edges incident to node $u_q$ in $S$. Similarly, let $F'_{u_q}$ and $F'_{u_{q+1}}$ be the sets of edges incident to $u_q$ and $u_{q+1}$ in $S'$, respectively. 
The coefficients $c'_{f'}$ are defined as follows:
\begin{enumerate}
    \item[(a)] If $f' \notin F'_{u_q} \cup F'_{u_{q+1}}$, then $f'\in F$. Let $c'_{f'}:=c_{f'}$. 
    \item[(b)] If $f'\in F'_{u_q}$ and $f' \neq (u_q,u_{q+1})$, then $f'\in F$. Let $c'_{f'} := \frac{\sigma_{q+1} - \sigma_q}{\sigma_* -\sigma_q} c_{f'}$.
    
    \item[(c)] If $f'= (u_i, u_{q+1})$ and $u_i \neq u_q$, then we pick the $f \in F$ such that 
    $$f =\begin{cases} (u_i, u_q) & \mbox{ if }u_i \neq u_{q+1}, \\
    (u_q,u_q) & \mbox{ if } u_i = u_{q+1}.
    \end{cases}
    $$ 
   Let $c'_{f'} := \frac{\sigma_* - \sigma_{q+1}}{\sigma_* -\sigma_q} c_{f}$. 
    
    \item[(d)] If $f' = (u_q,u_{q+1})$, then let $c'_{f'}:=0$.
\end{enumerate}


With the coefficients as above, we prove entry-wise that~\eqref{eq:x'c'f'z'} holds. 
First, note that because we obtain $S'$ by splitting the last node $u_q$ of $S$, the $i$th entry of $x'$, for $1\leq i\leq q-1$, is equal to $x_i$, so $x'_{i}=x_i = (\sigma_{i} - \sigma_{i-1})$. For the $i$th entry of the right hand side of~\eqref{eq:x'c'f'z'}, we consider two cases: 
\vspace{.1cm}

\noindent
{\em Case 1: $u_{i}$ is not incident to $u_q$ in $S$.} In this case, $u_i$ is not incident to either $u_q$ or $u_{q+1}$ in $S'$. Consequently, $F'_{u_{i}} = F_{u_{i}}$ and $z'_{f,i} = z_{f,i}$ for all $f\in F_{u_i}$. Furthermore, by item (a), $c'_{f} = c_f$ for any $f\in F_{u_{i}}$.  
Thus, the $i$th entry of the right hand side of~\eqref{eq:x'c'f'z'} is given by
$$
\sum_{f'\in F'_{u_{i}}} c'_{f'} z'_{f',i} = \sum_{f\in F_{u_{i}}}  c_f z_{f,i} = x_{i}  = \sigma_{i} - \sigma_{i-1}.
$$
\vspace{.1cm}

\noindent
{\em Case 2: $u_{i}$ is incident to $u_q$ in $S$.} In this case, $u_i$ is incident to both $u_q$ and $u_{q+1}$ in~$S'$. 
Let $g':= (u_i,u_q)$ and $h':= (u_i,u_{q+1})$ be the corresponding edges in~$S'$, see Fig.~\ref{fig:split} for an illustration. 
Then, the $i$th entry of the right hand side of~\eqref{eq:x'c'f'z'} is given by
\begin{equation}\label{eq:ithentryofx'}
\sum_{f'\in F'_{u_{i}}} c'_{f'} z'_{f',i}  = c'_{g'} z'_{g',i} + c'_{h'} z'_{h',i}+ \sum_{f'\in F'_{u_i} - \{g', h'\}} c'_{f'} z'_{f',i}.
\end{equation}
By items (b) and (c), 
$$c'_{g'} = \frac{\sigma_{q+1} - \sigma_q}{\sigma_* - \sigma_q} c_{g'} \quad \mbox{and}  \quad c'_{h'} = \frac{\sigma_* - \sigma_{q+1}}{\sigma_* - \sigma_q} c_{h'}.$$ 
Also, note that  
$$z'_{g',i}=z'_{h',i}= z_{g',i} = \frac{1}{2}.$$
Thus, the sum of the first two terms on the right hand side of~\eqref{eq:ithentryofx'} is $c_{g'} z_{g',i}$. For the last term,  note that $F'_{u_i} - \{g',h'\}=F_{u_i} - \{g'\}$. Also, by item (a) and the fact that $z'_{f,i} = z_{f,i}$ for any $f\in F_{u_i}$,  
$$
\sum_{f'\in F'_{u_i} - \{g',h'\}} c'_{f'} z'_{f',i} = \sum_{f\in F_{u_i} -\{g'\}} c_{f} z_{f,i}. 
$$
Combining the above arguments, we have that the right hand side of~\eqref{eq:ithentryofx'} is given by
$$
\sum_{f\in F_{u_{i}}} c_f z_{f,i} = x_{i} = \sigma_{i} - \sigma_{i-1}.
$$

Next, the $q$th entry of $x'$ is $(\sigma_{q+1} - \sigma_q)$ and the $q$th entry of the right hand side of~\eqref{eq:x'c'f'z'} is 
$$
\sum_{f'\in F'_{u_q}} c'_{f'} z'_{f',q} = \frac{\sigma_{q+1} - \sigma_q}{\sigma_* -\sigma_q}\sum_{f\in F_{u_q}}  c_f z_{f,q} = \frac{\sigma_{q+1} - \sigma_q}{\sigma_* -\sigma_q} x_q =   \sigma_{q+1} -\sigma_q,
$$
where the first equality follows from the fact that 
$$F'_{u_q} = F_{u_q} \cup \{(u_q,u_{q+1}) \mbox{ if } (u_q,u_q)\in F \},$$ 
items (b) and (d), and the last equality follows from the fact that $x_q = \sigma_{*} - \sigma_q$.

The last (i.e., $(q+1)$th) entry of $x'$ is $(\sigma_* - \sigma_{q+1})$. The last entry of the right hand side of~\eqref{eq:x'c'f'z'} is given by 
$$
\sum_{f'\in F'_{u_{q+1}}} c'_{f'} z'_{f',q+1} = \frac{\sigma_* - \sigma_{q+1}}{\sigma_* -\sigma_q}\sum_{f\in F_{u_q}} c_f z_{f,q} = \frac{\sigma_* - \sigma_{q+1}}{\sigma_* -\sigma_q} x_q = \sigma_* -\sigma_{q+1},
$$
where the first equality follows from item (c) above.  
We have thus shown that Eq.~\eqref{eq:x'c'f'z'} holds. In particular, since $c'_{f'}$ are nonnegative by construction, Eq.~\eqref{eq:x'c'f'z'} implies that $x'\in \cX(S')$.

It now remains to show that if $x \in \rint \cX(S)$, then $x' \in \rint \cX(S')$. Assuming $x \in \rint \cX(S)$,  if $u_q$ does not have a self-loop in $S$, then the edge $(u_q,u_{q+1})$ does not exist in $S'$, so by items (a), (b), and (c), all coefficients $c'_{f'}$ are positive, which implies that $x' \in \rint \cX(S')$.

We now assume that $u_q$ has a self-loop in $S$. 
Then, $k':= (u_q,u_{q+1})$ is an edge in $S'$ (see Fig.~\ref{fig:split} for an illustration), and thus $c'_{k'}=0$ per item (d) above. 
In this case, both $u_q$ and $u_{q+1}$ have self-loops in $S'$. Denote these two self-loops by $\ell'_{q} := (u_q,u_q)$ and $\ell'_{q+1} := (u_{q+1}, u_{q+1})$. 
By~\eqref{eq:defZS}, we have that 
$$z'_{k'}= \frac{1}{2} (z'_{\ell'_{q}} +  z'_{\ell'_{q+1}}).$$ 
Since $c'_{\ell'_{q}}$ and $c'_{\ell'_{q+1}}$ are positive, there exists an $\epsilon> 0$ such that $\epsilon < c'_{\ell'_{q}}$ and $\epsilon < c'_{\ell'_{q+1}}$. It then follows that
\begin{equation}\label{eq:manyprimes}
c'_{\ell'_{q}} z'_{\ell'_{q}} + c'_{\ell'_{q+1}} z'_{\ell'_{q+1}} = 2\epsilon z'_{k'} + (c'_{\ell'_{q}} - \epsilon) z'_{\ell'_{q}} + (c'_{\ell'_{q+1}} - \epsilon) z'_{\ell'_{q+1}}.
\end{equation}
Plugging in~\eqref{eq:x'c'f'z'} the relation~\eqref{eq:manyprimes} shows that $x'$ can be written as a convex combination of the $z'_{f'}$, for $f'\in F'$, with all {\em positive} coefficients,  and thus  $x' \in \rint \cX(S')$.
\vspace{.1cm}

\noindent{\bf Part 2: $x'\in \cX(S')\Rightarrow  x\in \cX(S)$ ($x'\in \rint \cX(S')\Rightarrow  x\in \rint \cX(S)$). } 
Because $x'\in \cX(S')$ (resp. $x'\in \rint \cX(S')$), we can write $x' = \sum_{f'\in F'}c'_{f'} z'_{f'}$, with $c'_{f'} \geq 0$ (resp. $c'_{f'} > 0$), for all $f'\in F'$. We will use $c'_{f'}$ to construct $c_{f}$, for $f\in F$, so that 
\begin{equation}\label{eq:xczf}
x = \sum_{f\in F} c_f z_f.
\end{equation}
To this end, we define $c_f$ as follows:
\begin{enumerate}
    \item[(e)] If $f$ is not incident to $u_q$ in $S$, 
    then let $c_f := c'_f$. 
    \item[(f)] If $f = (u_i,u_q)$ and $u_i\neq u_q$, then $g':=(u_i,u_q)$ and $h':=(u_i,u_{q+1})$ are edges in $S'$, and let $c_f := c'_{g'} + c'_{h'}$.
    \item[(g)] If $f = (u_q,u_q)$, then  $k':=(u_q,u_{q+1})$, $\ell'_{q} := (u_q,u_q)$, and $\ell'_{q+1} := (u_{q+1},u_{q+1})$ are edges in $S'$, and let $  c_f := c'_{k'}+ c'_{\ell'_{q}}+c'_{\ell'_{q+1}}$. 
\end{enumerate}
Note that all the coefficients $c_f$, for $f\in F$, defined above are nonnegative. Further, if all the $c'_{f'}$ are positive, i.e., $x'\in \rint \cX(S')$, then the $c_{f}$ are positive as well, which implies $x\in \rint \cX(S)$ provided that~\eqref{eq:xczf} holds.

We now show that the coefficients given above are such that~\eqref{eq:xczf} indeed holds. We do so by checking that~\eqref{eq:xczf}  holds for each entry.

For the $i$th entry, with $1 \leq i <q$, the left hand side of~\eqref{eq:xczf} is $x_i = (\sigma_i - \sigma_{i-1})$. For the right-hand side, if $(u_i,u_q)$ is an edge in $S$, then $g'$ and $h'$, as defined item~(f), are two edges in $S'$ and, consequently,  $F'_{u_i} = F_{u_i} \cup \{h'\}$.  
Note that $z_{f,i} = z'_{f,i}$ for all $f\in F_{u_i}$ and $$z_{g',i} = z'_{g',i} = z'_{h',i} = \frac{1}{2}.$$
Thus, by items (e) and (f), we have that
\begin{align*}
\sum_{f \in F} c_f z_{f,i} & =  c_{g'} z_{g',i} + \sum_{f \in F_{u_i}-\{g'\} }c_{f}z_{f,i} \\
& = c'_{g'} z'_{g',i} + c'_{h'} z'_{h',i} + \sum_{f'\in F'_{u_i} - \{g', h'\}} c'_{f'} z'_{f',i} \\
& = \sum_{f'\in F'_{u_i}} c'_{f'} z'_{f',i} = x'_{i} = (\sigma_{i} - \sigma_{i-1}).
\end{align*}

Finally, for the last entry, i.e., the $q$th entry, the left hand side of~\eqref{eq:xczf} is $x_q= \sigma_*-\sigma_q$. 
For the right hand side of~\eqref{eq:xczf}, we let $\ell_{q} := (u_q,u_q)$ be the loop on $u_q$ (if it exists in $S$) and thus have that
\begin{equation}\label{eq:firstentry}
\sum_{f\in F_{u_q}} c_f z_{f,q}  = c_{\ell_q} z_{\ell_{q},q} + \sum_{f\in F_{u_q} - \{\ell_{q}\}} c_f z_{f,q}. 
\end{equation}
Let $k'$, $\ell'_q$, and $\ell'_{q+1}$
be the three edges in $S'$ as defined in item~(g).  
Note that 
$$z_{\ell_{q},q} = z'_{\ell'_{q},q}=z'_{\ell'_{q+1},q+1}=2z'_{k',q} = 2z'_{k',q+1} = 1.$$
For the first term of~\eqref{eq:firstentry}, using item (g) and the above relations, we  obtain 
\begin{equation}\label{eq:cf11zf11}
c_{\ell_{q}} z_{\ell_{q},q} = c'_{\ell'_{q}} z'_{\ell'_{q},q} + c'_{\ell'_{q+1}}z'_{\ell'_{q+1},q+1} +c'_{k'}z'_{k',q} + c'_{k'}z'_{k',q+1}.
\end{equation}
For each addend in the second term of~\eqref{eq:firstentry}, the edge $g  = (u_i,u_q)$ in $S$, for some $u_i\neq u_q$,  has two corresponding edges in $S'$, namely  
$g' = (u_i,u_q)$ and 
$h' = (u_i,u_{q+1})$. 
Note that
$$
z_{g,q} = z'_{g',q} = z'_{h',q+1} = \frac{1}{2}.
$$
Then, by item (f), 
\begin{equation}\label{eq:cf1izf1i}
c_{g} z_{g,q} = c'_{g'} z'_{g',q} + c'_{h'}z'_{h',q+1}.  
\end{equation}
Combining~\eqref{eq:cf11zf11} and~\eqref{eq:cf1izf1i}, we obtain that
\begin{align}
\sum_{f\in F_{u_q}} c_f z_{f,1} & = \sum_{f'\in F'_{u_q}} c'_{f'} z'_{f',q}  + 
\sum_{f'\in F'_{u_{q+1}}} c'_{f'} z'_{f',q+1}\notag \\ 
& = x'_q + x'_{q+1} = (\sigma_{q+1} - \sigma_q) + (\sigma_* - \sigma_{q+1}) = \sigma_* - \sigma_q.  \notag
\end{align}
This concludes the proof.
\end{proof}

\end{document}